\documentclass[11pt]{amsart}

\DeclareMathOperator{\Lie}{Lie}

\DeclareMathOperator{\cl}{cl}

\DeclareMathOperator{\Spec}{Spec}

\DeclareMathOperator{\Char}{char}
\usepackage[usenames,dvipsnames]{color}
\usepackage[normalem]{ulem}
\usepackage{amsmath,amssymb,amsthm,graphics,amscd,mathrsfs}
\usepackage{longtable}
\usepackage{cite}
\usepackage{color}
\usepackage{graphicx}
\usepackage{epsf}
\usepackage{fancyhdr,hyperref}
\usepackage{lscape}
\usepackage{verbatim}

\hypersetup{colorlinks=true,citecolor=Purple}

  \renewenvironment{thebibliography}[1]{
    \begin{oldthebibliography}{#1}
      \setlength{\parskip}{0ex}
      \setlength{\itemsep}{0ex}
  }
  {
    \end{oldthebibliography}
  }
  
\setlongtables
\begin{document}

\newcounter{rownum}
\setcounter{rownum}{0}
\newcommand{\ab}{\addtocounter{rownum}{1}\arabic{rownum}}

\newcommand{\x}{$\times$}
\newcommand{\bb}{\mathbf}

\newcommand{\Ind}{\mathrm{Ind}}
\newcommand{\R}{\mathrm{R}}
\newcommand{\RR}{\mathscr{R}}
\newcommand{\G}{\mathscr{G}}
\newcommand{\hra}{\hookrightarrow}
\newcommand{\sss}{\mathrm{ss}}
\newtheorem{lemma}{Lemma}[section]
\newtheorem{theorem}[lemma]{Theorem}
\newtheorem*{TA}{Theorem A}
\newtheorem*{TB}{Theorem B}
\newtheorem*{TC}{Theorem C}
\newtheorem*{CorC}{Corollary C}
\newtheorem*{TD}{Theorem D}
\newtheorem*{TE}{Theorem E}
\newtheorem*{PF}{Proposition E}
\newtheorem*{C3}{Corollary 3}
\newtheorem*{T4}{Theorem 4}
\newtheorem*{C5}{Corollary 5}
\newtheorem*{C6}{Corollary 6}
\newtheorem*{C7}{Corollary 7}
\newtheorem*{C8}{Corollary 8}
\newtheorem*{claim}{Claim}
\newtheorem{cor}[lemma]{Corollary}
\newtheorem{conjecture}[lemma]{Conjecture}
\newtheorem{prop}[lemma]{Proposition}
\newtheorem{question}[lemma]{Question}
\theoremstyle{definition}
\newtheorem{example}[lemma]{Example}
\newtheorem{examples}[lemma]{Examples}
\newtheorem{algorithm}[lemma]{Algorithm}
\newtheorem*{algorithm*}{Algorithm}
\theoremstyle{remark}
\newtheorem{remark}[lemma]{Remark}
\newtheorem{remarks}[lemma]{Remarks}
\newtheorem{obs}[lemma]{Observation}
\theoremstyle{definition}
\newtheorem{defn}[lemma]{Definition}

  \def\hal{\unskip\nobreak\hfil\penalty50\hskip10pt\hbox{}\nobreak
  \hfill\vrule height 5pt width 6pt depth 1pt\par\vskip 2mm}

\renewcommand{\labelenumi}{(\roman{enumi})}
\newcommand{\Hom}{\mathrm{Hom}}
\newcommand{\Int}{\mathrm{int}}
\newcommand{\Ext}{\mathrm{Ext}}
\newcommand{\opH}{\mathrm{H}}
\newcommand{\D}{\mathscr{D}}
\newcommand{\soc}{\mathrm{Soc}}
\newcommand{\SO}{\mathrm{SO}}
\newcommand{\Sp}{\mathrm{Sp}}
\newcommand{\SL}{\mathrm{SL}}
\newcommand{\GL}{\mathrm{GL}}
\newcommand{\PGL}{\mathrm{PGL}}
\newcommand{\OO}{\mathcal{O}}
\newcommand{\Y}{\mathbf{Y}}
\newcommand{\X}{\mathbf{X}}
\newcommand{\diag}{\mathrm{diag}}
\newcommand{\End}{\mathrm{End}}
\newcommand{\tr}{\mathrm{tr}}
\newcommand{\Stab}{\mathrm{Stab}}
\newcommand{\red}{\mathrm{red}}
\newcommand{\Aut}{\mathrm{Aut}}
\renewcommand{\H}{\mathcal{H}}
\renewcommand{\u}{\mathfrak{u}}
\newcommand{\Ad}{\mathrm{Ad}}
\newcommand{\N}{\mathcal{N}}
\newcommand{\Z}{\mathbb{Z}}
\newcommand{\la}{\langle}\newcommand{\ra}{\rangle}
\newcommand{\gl}{\mathfrak{gl}}
\newcommand{\g}{\mathfrak{g}}
\newcommand{\F}{\mathbb{F}}
\newcommand{\m}{\mathfrak{m}}
\renewcommand{\b}{\mathfrak{b}}
\newcommand{\p}{\mathfrak{p}}
\newcommand{\q}{\mathfrak{q}}
\renewcommand{\l}{\mathfrak{l}}
\newcommand{\del}{\partial}
\newcommand{\h}{\mathfrak{h}}
\renewcommand{\t}{\mathfrak{t}}
\renewcommand{\k}{\mathfrak{k}}
\newcommand{\Gm}{\mathbb{G}_m}
\renewcommand{\c}{\mathfrak{c}}
\renewcommand{\r}{\mathfrak{r}}
\newcommand{\n}{\mathfrak{n}}
\newcommand{\s}{\mathfrak{s}}
\newcommand{\Q}{\mathbb{Q}}
\newcommand{\z}{\mathfrak{z}}
\newcommand{\pso}{\mathfrak{pso}}
\newcommand{\so}{\mathfrak{so}}
\renewcommand{\sl}{\mathfrak{sl}}
\newcommand{\psl}{\mathfrak{psl}}
\renewcommand{\sp}{\mathfrak{sp}}
\newcommand{\Ga}{\mathbb{G}_a}
\newcommand{\barB}{\overline{B}}
\newcommand{\barm}{\overline{\mathfrak{m}}}

\newenvironment{changemargin}[1]{%
  \begin{list}{}{%
    \setlength{\topsep}{0pt}%
    \setlength{\topmargin}{#1}%
    \setlength{\listparindent}{\parindent}%
    \setlength{\itemindent}{\parindent}%
    \setlength{\parsep}{\parskip}%
  }%
  \item[]}{\end{list}}

\parindent=0pt
\addtolength{\parskip}{0.5\baselineskip}

\title{On unipotent radicals of pseudo-reductive groups}

\author[M.\  Bate]{Michael Bate}
\address
{Department of Mathematics,
University of York,
York YO10 5DD,
United Kingdom}
\email{michael.bate@york.ac.uk}
\author[B.\ Martin]{Benjamin Martin}
\address%[B.\ Martin]
{Department of Mathematics,
University of Aberdeen,
King's College,
Fraser Noble Building,
Aberdeen AB24 3UE,
United Kingdom}
\email{b.martin@abdn.ac.uk}
\author[G. R\"ohrle]{Gerhard R\"ohrle}
\address
{Fakult\"at f\"ur Mathematik,
Ruhr-Universit\"at Bochum,
D-44780 Bochum, Germany}
\email{gerhard.roehrle@rub.de}
\author{David I. Stewart}
\address{School of Mathematics and Statistics,
Herschel Building,
Newcastle,
NE1 7RU, UK}
\email{david.stewart@ncl.ac.uk}

\pagestyle{plain}
\begin{abstract}
We establish some results on the structure of the geometric unipotent radicals of pseudo-reductive $k$-groups. In particular, our main theorem gives bounds on the nilpotency class of geometric unipotent radicals of standard pseudo-reductive groups, which are sharp in many cases. 
A major part of the proof rests upon consideration of the following situation: let $k'$ be a purely inseparable field extension of $k$ of degree $p^e$ and let $G$ denote the Weil restriction of scalars $\R_{k'/k}(G')$ of a reductive $k'$-group $G'$. When $G= \R_{k'/k}(G')$ we also provide some results on the orders of elements of the unipotent radical $\RR_u(G_{\bar k})$ of the extension of scalars of $G$ to the algebraic closure $\bar k$ of $k$.
\end{abstract}
\maketitle
%{\small \tableofcontents}
\section{Introduction}
\label{sec:intro}

Let $G$ be a smooth affine algebraic $k$-group over an arbitrary field $k$. Then $G$ is said to be \emph{pseudo-reductive} if $G$ is connected and the largest $k$-defined connected smooth normal unipotent subgroup $\RR_{u,k}(G)$ of $G$ is trivial. J. Tits introduced pseudo-reductive groups to the literature some time ago in a series of courses at the Coll\`ege de France (\!\!\cite{Tits9192} and \cite{Tits9293}), but they have resurfaced rather dramatically in recent years thanks to the monograph \cite{CGP15}, many of whose results were used in B.~Conrad's proof of the finiteness of the Tate--Shafarevich sets and Tamagawa numbers of arbitrary linear algebraic groups over global function fields \cite[Thm.\ 1.3.3]{Con12}. The main result of that monograph is \cite[Thm.~5.1.1]{CGP15} which says that unless one is in some special situation over a field of  characteristic $2$ or $3$, then any pseudo-reductive group is \emph{standard}. This means it arises after a process of modification of a Cartan subgroup of a certain Weil restriction of scalars of a given reductive group (we assume reductive groups are connected). More specifically, a standard pseudo-reductive group $G$ can be expressed as a quotient group of the form \[G= (\R_{k'/k}(G')\rtimes C)/\R_{k'/k}(T')\] corresponding to a $4$-tuple $(G',k'/k,T',C)$, where $k'$ is a non-zero finite reduced $k$-algebra, $G'$ is a $k'$-group with reductive fibres over $\Spec k'$, $T'$ is a maximal $k'$-torus of $G'$ and $C$ is a commutative pseudo-reductive $k$-group occurring in a factorisation
\[\R_{k'/k}(T')\stackrel{\phi}{\to}C\stackrel{\psi}{\to}\R_{k'/k}(T'/Z_{G'})\tag{$*$}\]
of the natural map $\varpi:\R_{k'/k}(T')\to \R_{k'/k}(T'/Z_{G'})$.
Here $Z_{G'}$ is the (scheme-theoretic) centre of $G'$\footnote{Note that $\varpi$ is {\em not} surjective when $Z_{G'}$ has non-\'etale fibre at a factor field $k'_i$ of $k'$ that is not separable over $k$.} and $C$ acts on $\R_{k'/k}(G')$ via $\psi$ followed by the functor $\R_{k'/k}$ applied to the conjugation action of $T'/Z_{G'}$ on $G'$; we regard $\R_{k'/k}(T')$ as a central subgroup of $\R_{k'/k}(G')\rtimes C$ via the map $h\mapsto (i(h)^{-1}, \phi(h))$, where $i$ is the natural inclusion of $\R_{k'/k}(T')$ in $\R_{k'/k}(G')$ .

The structure of general connected linear algebraic groups over perfect fields $k$ is well-understood: the geometric unipotent radical $\RR_u(G_{\bar k})=\RR_{u,\bar k}(G_{\bar k})$ descends to a subgroup $\RR_u(G)$ of $G$, and the quotient $G^\red:=G/\RR_u(G)$ is reductive. If one further insists $k$ be separably closed one even has that $G^\red$ is split, so $G^\red$ is the central product of its semisimple derived group $\D(G)$ and a central torus, with $G^\red/Z_G$ semisimple---in fact, the direct product of its simple factors.

Most of this theory goes wrong over imperfect fields $k$, hence in particular the need to consider pseudo-reductive groups, whose geometric unipotent radicals may not be defined over $k$. In order to understand the structure of a given smooth affine algebraic group $G$ over $k$ it is therefore instructive to extend scalars to the (perfect) algebraically closed field $\bar k$ and analyse the structure of $G_{\bar k}$, where, for example, one sees the full unipotent radical. We pursue this approach in this paper and discuss the structure of $G_{\bar k}$ where $G$ is a standard pseudo-reductive group arising from a $4$-tuple $(G',k'/k,T',C)$. The reductive part $G_{\bar k}^\red=G_{\bar k}/\RR_u(G_{\bar k})$ is not especially interesting in that the universal property of Weil restriction implies that $G_{\bar k}^\red$ has the same root system as $G'$. Further results \cite[Thm.~3.4.6, Cor.~A.5.16]{CGP15} even furnish us, under some restrictive conditions, with a Levi subgroup for $G$: a smooth subgroup $H$ such that $H_{\bar k}$ is a complement to the geometric unipotent radical $\RR_u(G_{\bar k})$. However, the precise structure of $\RR_u(G_{\bar k})$ is rather more mysterious. While one knows that there is a composition series of $\RR_u(G_{\bar k})$ whose composition factors are related to the adjoint $G'_{\bar k}$-module $\g'=\Lie(G'_{\bar k})$ (see Lemma~\ref{lem:mofo}), it is unclear what the structure of $\RR_u(G_{\bar k})$ is \emph{qua} group.  Since $\RR_u(G_{\bar k})$ is a $p$-group, one may consider some standard invariants, which measure the order of its elements and the extent to which it is non-abelian. One major purpose of this paper is to show that as soon as the root system associated to $G$ is non-trivial and $k'/k$ is finite and inseparable, then $\RR_u(G_{\bar k})$ is almost always non-abelian (in a way that depends on the characteristic of the field amongst other things).

To state our results, let $(G',k'/k,T',C)$ be a $4$-tuple giving rise to a pseudo-reductive group $G$ and denote the separable closure of $k$ by $k_s$. Then $k_s'=k'\otimes_k {k_s}$ is a finite non-zero reduced $k_s$-algebra, isomorphic to a product of factor fields $\prod_i k_i$, where each $k_i/k_s$ is a purely inseparable extension. Let $G_i'$ (resp., $T_i$) denote the fibre of $G'_{k_s'}$ (resp., $T'_{k_s}$) over $k_i$. By standard results, $\R_{k'/k}(G')_{k_s}\cong\R_{k_s'/k_s}(G')$ (see Lemma \ref{reducetosplit}). 
If \[p=2 \quad\text{and}\quad G_i'\cong \SL_2^r\times S',\] for $S'$ a torus and $r>0$, we say $G_i'$ is \emph{unusual}. We define a number $\ell_i$ depending on $G_i'$ and $k_i$. First, if $G_i'$ is commutative, then set $\ell_i=1$. Now, for each field $k_i$, form the $k_s$-algebra $B_i=k_i\otimes_{k_s} k_i$. Each $B_i$ is a local $k_i$-algebra, which by virtue of the pure inseparability of $k_i/k_s$ means that its maximal ideal, $\m_i$ say, consists of nilpotent elements. Hence there is a minimal $n$ such that $\m_i^n=0$; in case $G_i'$ is not unusual, let $\ell_i:=n-1$. In case $G_i'$ is unusual, then define $\ell_i$ to be the minimal $n$ such that
$({}^2\m_i)^{n-1}\cdot \m_i^2=0$
where $({}^2\m_i)$ is the ideal of squares in $\m_i$. Finally  define the integer \[N:={\max}_i\left\{\ell_i\right\}.\]

\begin{theorem}\label{maintheorem}Let $G$ be a non-commutative standard pseudo-reductive group over a field $k$ of characteristic $p$, arising from a $4$-tuple $(G',k'/k,T',C)$ with $k'$ a non-zero finite reduced $k$-algebra. Let the $G_i'$ (resp., the $T_i'$) be the fibres of $G'$ (resp., $T'$) over the factor fields of $k_s$, as above.  Suppose that for every unusual fibre $G_i'$, the map $\phi$ in the factorisation ($*$) is an injection on restriction to $T'_i\cap G_i'$. Then the nilpotency class of $\RR_u(G_{\bar k})$ is $N$.
\end{theorem}

What the theorem indicates is that the precise structure of $\RR_u(G_{\bar k})$, including its nilpotency class, appears to depend in a very particular way on the nature of the extension $k'/k$ used in Weil restriction, rather than on, for instance, the reductive group $G'$ (provided the reductive group contains at least one fibre that is not unusual or commutative).

\begin{remarks}(i) There is one set of cases we have excluded. After passing to an appropriate direct factor of a fibre of $G'$ over $k'$, this happens when $G$ arises from a $4$-tuple $(G',k'/k,T',C)$, where $G'=\SL_2$.  The kernel of the map $\R_{k'/k}(T')\to \R_{k'/k}(T'/Z_{G'})$ is $\R_{k'/k}(\mu_2)$ (see \cite[Prop.\ 1.3.4]{CGP15}).  For example, let us suppose in the factorisation $\R_{k'/k}(T')\stackrel{\phi}{\to} C\stackrel{\psi}{\to}\R_{k'/k}(T'/Z_{G'})$ that $\ker(\phi)= \R_{k'/k}(\mu_2)$ and $\psi$ is injective.  Then we may regard $C$ as a subgroup of $\R_{k'/k}(T'/Z_{G'})$ sitting in between $\R_{k'/k}(T')/\R_{k'/k}(\mu_2)$ and $\R_{k'/k}(T'/Z_{G'})$.  If $C= \R_{k'/k}(T')/\R_{k'/k}(\mu_2)$ then $G\cong \R_{k'/k}(\SL_2)/\R_{k'/k}(\mu_2)$, and the nilpotency class of the unipotent radical of $G_{\bar k}$ is at most the minimal $n$ such that $({}^2\m_i)^{n-1}\cdot \m_i^2=0$ (see Remark~\ref{rem:weirdquot})---in other words it is bounded above by the nilpotency class of the unipotent radical of $\R_{k'/k}(\SL_2)_{\bar k}$.  If $C= \R_{k'/k}(T'/Z_{G'})$ then  $G\cong \R_{k'/k}(\PGL_2)$ (cf.\ Remark~\ref{rem:badstd}), so the unipotent radical of $G_{\bar k}$ has nilpotency class $n-1$ for the minimal $n$ such that $\m^n=0$. These numbers are usually different. In between these two extremes there is a  range of possibilities: there may be many subgroups $C$ sitting in between $\R_{k'/k}(\SL_2)/\R_{k'/k}(\mu_2)$ and $\R_{k'/k}(T'/Z_{G'})$---note that the quotient $\R_{k'/k}(T'/Z_{G'})/(\R_{k'/k}(\SL_2)/\R_{k'/k}(\mu_2))$ is a smooth commutative unipotent group---and we have, \emph{a priori}, no control over these.

(ii) Given any standard pseudo-reductive group $G$, by \cite[Thm.~4.1.1]{CGP15},  
we may find a \emph{standard presentation} for it; that is, a $4$-tuple $(G',k'/k,T',C)$ giving rise to $G$ through the standard construction such that $G'$ has absolutely simple, simply-connected fibres over $\Spec k'$. In this case, one may simplify the statement of the theorem somewhat, since then a fibre of $G'_i$ of $G'$ is unusual precisely when $(G'_i)_{k_s}\cong \SL_2$ and $p=2$.\end{remarks}

To prove the theorem, we reduce to the case that $G$ is a Weil restriction, but in order to deal with cases where the centre $Z_{G'}$ of $G'$ is not smooth, we show how to present $G$ through the standard construction, starting from a central extension $\widehat{G'}$ of $G'$ such that the centre of $\widehat{G'}$ is smooth (see Section~\ref{subsec:hatty}). This may be of independent interest.

In Section \ref{sec:orders} we consider for $G$ of the form $\R_{k'/k}(G')$ another invariant of $\RR_u(G_{\bar k})$ usually attached to $p$-groups---the \emph{exponent}, i.e., the minimal integer $s$ such that the $p^s$-power map on $\RR_u(G_{\bar k})$ factors through the trivial group. For this invariant, we do not have as precise a result as we do for the nilpotency class, but we establish some upper and lower bounds, and identify a possible form of a precise result, cf.\ Question \ref{opq}.

Finally, let us remark that part of the proof of Theorem~\ref{maintheorem} reduces to some bare-hands calculations with matrices arising from the Weil restriction $\R_{k'/k}(\GL_2)$ and may be of interest to anyone who would like to see some examples of pseudo-reductive groups in an explicit description by matrices.

\section{Notation and Preliminaries}

We follow the notation of \cite{CGP15}. In particular, $k$ is always a field of characteristic $p> 0$.  All algebraic groups are assumed to be affine and of finite type over the ground ring, and all subgroups are closed.  Reductive and pseudo-reductive groups are assumed to be smooth and connected.

Let us recall the definition of Weil restriction and some relevant features. We consider algebraic groups scheme-theoretically, so that an algebraic $k$-group $G$ is a functor $\{k\textrm{-algebras}\to\textrm{groups}\}$ which is representable via a finitely-presented $k$-algebra $k[G]$. Let $k'$ be a non-zero finite reduced $k$-algebra. Then for any smooth $k'$-group $G'$ with connected fibres over $\Spec k'$, the Weil restriction $G=\R_{k'/k}(G')$ is a smooth connected $k$-group of dimension $[k':k]\dim G'$, characterised by the property $G(A)=G'(k'\otimes_k A)$ functorially in $k$-algebras $A$. If $H'$ is a subgroup of $G'$ then $\R_{k'/k}(H')$ is a subgroup of $\R_{k'/k}(G')$.  For a thorough treatment of this, one may see \cite[A.5]{CGP15}. Important for us is the fact that Weil restriction is right adjoint to base change: that is, we have a bijection
\[\Hom_k(M,\R_{k'/k}(G'))\cong \Hom_{k'}(M_{k'},G')\]
natural in the $k$-group scheme $M$. We need two special cases.  First, if $M=G=\R_{k'/k}(G')$ then the identity morphism $G\to G$ corresponds to a map $q_{G'}:G_{k'}\to G'$.  When $k'$ is a finite purely inseparable field extension of $k$ and $G'$ is reductive then by \cite[Thm.\ 1.6.2]{CGP15}, $q_{G'}$ is smooth and surjective, $\ker q_{G'}$ coincides with $\RR_{u,k'}(G_{k'})$ and $\RR_{u,k'}(G_{k'})$ is a descent of $\RR_{u}(G_{\bar k})$ (so $\RR_{u}(G_{\bar k})$ is defined over $k'$); it follows from the naturality of the maps concerned that if $H'$ is a connected reductive subgroup of $G'$ then $\ker q_{H'}\subseteq \ker q_{G'}$, so $\RR_{u,k'}(H_{k'})\subseteq \RR_{u,k'}(G_{k'})$.  In particular, $\dim \RR_{u,k'}(G_{k'})= ([k':k]- 1)\dim G'$.  Second, if $H$ is a reductive $k$-group and $M= G'= H_{k'}$ then the identity morphism $H_{k'}\to H_{k'}$ corresponds to a map $i_H:H\to \R_{k'/k}(H_{k'})$.  When $k'$ is a finite purely inseparable field extension of $k$ then by \cite[Cor.\ A.5.16]{CGP15}, the composition $H_{k'}\stackrel{(i_H)_{k'}}{\longrightarrow} \R_{k'/k}(H_{k'}) \stackrel{q_{H_{k'}}}{\longrightarrow} H_{k'}$ is the identity, so we may regard $H$ as a subgroup of $\R_{k'/k}(H_{k'})$ via $i_H$; in fact, $H$ is a Levi subgroup of $\R_{k'/k}(H_{k'})$.

Let $k'=\prod_{i=1}^n k_i$ be a non-zero finite reduced $k$-algebra with factor fields $k_i$.
Then any algebraic $k'$-group $G'$ decomposes as a product $G' = \prod_{i=1}^n G'_i$
where each $G'_i$ is an algebraic $k_i$-group (it is the fibre of $G'$ over $\Spec k_i$).
In such a situation, we let $\RR_{u,k'}(G')$ denote the unipotent subgroup $\prod_{i=1}^n\RR_{u,k_i}(G_i)$.  
We have $\R_{k'/k}(G')= \prod_{i=1}^n \R_{k_i/k}(G'_i)$; in particular, $\RR_{u,k'}(\R_{k'/k}(G')_{k'})= \prod_{i=1}^n \RR_{u,k'}(\R_{k_i/k}(G'_i)_{k'}$.  This allows us to reduce immediately to the case that $k'$ is a field in the proof of Theorem~\ref{maintheorem}.  

Suppose now that $H$ is a smooth connected algebraic group over $k$ and let $[g,h]=ghg^{-1}h^{-1}$ denote the commutator of the elements $g,h\in H({\bar k})$.  Let $M,K$ be smooth connected subgroups of $H$.  Then the commutator subgroup $[M,K]$ is a smooth connected subgroup of $H$, and we may identify $[M,K]({\bar k})$ with the commutator subgroup $[M({\bar k}),K({\bar k})]$.  Moreover, $[M_{k'},K_{k'}]=([M,K])_{k'}$ for any field extension $k'/k$ (cf.\ \cite[I.2.4]{Bor91}).  In particular, we may form the lower central series $\{\D_m(H)\}_{m\geq 0}$ in the usual way, and $\D_m(H)({\bar k})$ is the $m^{\rm th}$ term in the lower central series for the abstract group $H({\bar k})$.  We say that $H$ is \emph{nilpotent} if there exists some integer $m$ such that $\D_m(H)=1$. The \emph{nilpotency class} $\cl(H)$ of $H$ is the smallest integer $m$ such that $\D_{m}(H)=1$.  The arguments above show that extending the base field does not change the nilpotency class of $H$.

In proving Theorem~\ref{maintheorem} and intermediate results, we usually want to reduce to the case that $k$ is separably closed, guaranteeing that $k_i/k$ is purely inseparable for $k_i$ a factor field of $k'$. This also allows us to assume that the group $G'$ has \emph{split} reductive fibres. 
We denote by $k_s$ the separable closure of $k$ in its algebraic closure $\bar k$, and we set $k_s'=k'\otimes_k k_s$, a non-zero finite reduced $k_s$-algebra.  Note that even when $k'$ is a field, $k_s'$ need not be a field.

\begin{lemma}\label{reducetosplit}Let $G$ be a standard pseudo-reductive group arising from $(G',k'/k,T',C)$. Then 

(i) $G_{k_s}$ is isomorphic to the standard pseudo-reductive group arising from  $(G'_{k_s'},k_s'/k_s,T'_{k_s'},C_{k_s})$.

(ii) $\RR_{u,k'}(G_{k'})_{k_s'}=\RR_{u,k_s'}(G_{k_s'})$.

\end{lemma}
\begin{proof}(i). We have $G=(\R_{k'/k}(G')\rtimes C)/\R_{k'/k}(T')$. In order to see that 
\[
G_{k_s}\cong(\R_{k_s'/k_s}(G'_{k_s'})\rtimes C_{k_s})/\R_{k_s'/k_s}(T'_{k_s'}) 
\]
it suffices to see that: (a) the sequences
\[1\to \R_{k'/k}(T')\to\R_{k'/k}(G')\rtimes C\to G\to 1\]
and
\[1\to \R_{k'/k}(G')\to\R_{k'/k}(G')\rtimes C\to C\to 1\]
remain exact after taking the base change to $k_s$, with the second remaining split; and (b) the formation of the Weil restriction $\R_{k'/k}(G')$ commutes with base change to $k_s$.

These facts are standard: for (a) this follows since algebras over a field $k$ are flat; and (b) is \cite[A.5.2(1)]{CGP15}.

(ii). Without loss we can assume $k'$ is a field.  Write $k_s'= \prod_{i\in I} k_i$.  Since \[\RR_{u,k'}(G_{k'})_{\bar k}=\RR_{u}(G_{\bar k})=\RR_u((G_{k_s})_{\bar k})=\RR_u((G_{k_i})_{\bar k})=\RR_{u,k_i}(G_{k_i})_{\bar k},\] we have $\RR_{u,k'}(G_{k'})_{k_i}=\RR_{u,k_i}(G_{k_i})$ for each $i\in I$.  The result now follows.
\end{proof}

\section{Proof of Theorem \ref{maintheorem}} 
\subsection{Upper bound} The following result provides an upper bound for the nilpotency class of $\RR_u(\R_{k'/k}(G')_{\bar k})$ in case $k'$ is a purely inseparable field extension. We will later reduce to this case.

\begin{prop}\label{lem:mofo} 
Let $G=\R_{k'/k}(G')$ for $k'$ a purely inseparable field extension of $k$. Then $B:=k'\otimes_k k'$ is local with maximal nilpotent ideal $\m$. Suppose $n$ is minimal such that $\m^n=0$. Then $\RR_u(G_{\bar k})$ has a filtration $\RR_u(G_{\bar k})=U_1\supseteq U_2\supseteq \dots\supseteq U_{n}=1$ satisfying the following:

(a) the successive quotients are $G'_{\bar k}$-equivariantly isomorphic to a direct sum of copies of the adjoint $G'_{\bar k}$-module;

(b) the commutator subgroup $[U_1,U_i]\subseteq U_{i+1}$ for any $i\geq 1$. In particular, we have $\cl(\RR_u(G_{\bar k}))\leq n-1$.
\end{prop}

\begin{proof}
Following \cite[A.3.6]{Oes84}, we can identify $G_{k'}=\R_{k'/k}(G')_{k'}$ with 
$\R_{B/k'}(G'_B)$ where $B:=k'\otimes_k k'$.
For the convenience of the reader, we reproduce some of the constructions from \cite[\S A.3.5]{Oes84} in this setting.

We begin with some basic observations and notation.
The maximal ideal $\m$ of $B$ is the kernel of the map $B\to k'$ which sends $b_1\otimes b_2 \to b_1b_2$,
so $\m$ is generated by elements of the form $b\otimes 1-1\otimes b$.
This is an ideal of nilpotent elements (everything is killed by a suitably high power of $p$),
so there is $n\in \mathbb{N}$ 
such that $\m^n = \{0\}$; choose $n$ minimal subject to this.
Note that $n\leq [k':k]$ since $B$ has dimension $[k':k]$ as a (left) vector space over $k'$,
so this is the largest possible size for a strictly decreasing chain of proper subspaces of $B$. 
Denote the quotient $B/\m^i$ for $1\leq i \leq n$ by $B_i$ and let $G_i:=\R_{B_i/k'}(G'_{B_i})$. 
Note that since $B_1 = B/\m \cong k'$ we have that $G_1 \cong G'$ and since $\m^n = \{0\}$ we have that  
$G_n \cong G_{k'}$. 
It is proved in \cite[Prop.~A.3.5]{Oes84} that the canonical surjection $B_i \to B_{i-1}$ for each $1\leq i\leq n$ gives rise to a surjective homomorphism of algebraic groups 
$G_{i} \to G_{i-1}$; let $V_i$ denote the kernel of this map for each $i$.
Let $A:= k'[G']\otimes_{k'} B$ denote the coordinate algebra of $G'_B$,
and let $A_i:= k'[G']\otimes_{k'} B_i$ denote the coordinate algebra of $G'_{B_i}$ for each $i$
(recall that $A_n = A$).

In \cite[Prop.~A.3.5]{Oes84}, Oesterl\'e observes that each $V_i$ is a vector group naturally isomorphic to the additive group $\Lie(G')\otimes_{k'} \m^{i-1}/\m^i$.
We recall how to see this.
Given any $k'$-algebra $R$,
the identity element of $G_i(k')$ gives a homomorphism of $B_i$-algebras $A_i\to B_i$ which we can compose with the canonical injection $B_i\to R\otimes_{k'} B_i$ to get a homomorphism $e_i: A_i\to R\otimes_{k'} B_i$. 
Then one notes that an $R$-point $v\in V_i(R)$ corresponds to a 
$B_i$-algebra homomorphism $v:A_i\to R\otimes_{k'} B_i$ whose difference with $e_i$ has image in 
$R\otimes_{k'} \m^{i-1}/\m^i$.

The ring $R\otimes_{k'} \m^{i-1}/\m^i$ is also an $A_i$-module via the identity element of $G_i(k')$:
multiplication by an element $a\in A_i$ under this action coincides with multiplication by $e_i(a)$
in $R\otimes_{k'} \m^{i-1}/\m^i$.
Now, given $a,b \in A_i$, we write
$$
(e_i-v)(ab) =  e_i(a)e_i(b) - v(a)v(b)  = (e_i-v)(a)e_i(b) + e_i(a)(e_i-v)(b) - (e_i-v)(a)(e_i-v)(b).
$$
The last term above is a product of terms in $R\otimes_{k'} \m^{i-1}/\m^i$, so is zero,
and hence we can conclude that $e_i-v$ acts like a $B_i$-linear derivation $\delta: A_i \to R\otimes_{k'} \m^{i-1}/\m^i$ (with the $A_i$-module structure we have given).
Conversely, any such derivation will gives rise to an element of $V_i(R)$ by reversing the above process.

Any $B_i$-linear derivation of $A_i$ in $R\otimes \m^{i-1}/\m^i$ will kill the ideal $(\m/\m^i) A_i$ in $A_i$ (since it is $B_i$-linear and multiplication by $\m/\m^i$ in $R\otimes_{k'} \m^{i-1}/\m^i$ kills everything); this means we can identify this set of derivations with the $k'$-linear derivations of $A_i/(\m/\m^i) A_i \cong k'[G']$ in $R\otimes \m^{i-1}/\m^i$---write $\textrm{Der}_{k'}(k'[G'],R\otimes_{k'}\m^{i-1}/\m^i)$ for this set.
Further, since the $A_i$-module structure on $R\otimes \m^{i-1}/\m^i$ comes from the identity element in $G_i(R) = G'(R\otimes B_i)$, we see that the induced $k'[G']$-module structure 
permits an identification
$$
\textrm{Der}_{k'}(k'[G'],R\otimes_{k'}\m^{i-1}/\m^i) \cong \textrm{Der}_{k'}(k'[G'],k')\otimes_{k'} R\otimes_{k'}\m^{i-1}/\m^i.
$$
Now the set of derivations on the right hand side is just $\Lie(G')$,
so we have the required identification of $V_i(R)$ with $\Lie(G') \otimes_{k'} R\otimes_{k'} \m^{i-1}/ \m^i$.

We now show the identifications in the previous paragraphs turn group multiplication in $V_i(R)$ into addition.
Given $v_1,v_2 \in V_i$, the idea is to rewrite 
$$
e_i-v_1v_2 = (e_i-v_1) + (e_i-v_2) - (e_i-v_1)(e_i-v_2)
$$ 
and then
to note that when we apply this equation to $a\in A_i$ the final term is zero in $R\otimes \m^{i-1}/\m^i$.
To justify this, given $a\in A_i$ let $\sum_{r,s} a_r\otimes a_s$ denote the image of $a$ under the comultiplication map $A_i\to A_i\otimes A_i$.
Then recalling that $e_i$ comes from the identity of $G_i(k')$, we can write:
\begin{align*}
(e_i-v_1v_2)(a) 
&= \sum_{r,s} (e_i(a_r)e_i(a_s)-v_1(a_r)v_2(a_s)) \\
&= \sum_{r,s} (e_i(a_r)e_i(a_s)-v_1(a_r)e_i(a_s)) + \sum_{r,s}(e_i(a_r)e_i(a_s)-e_i(a_r)v_2(a_s)) \\
&\qquad\qquad\qquad - \sum_{r,s}(e_i(a_r)-v_1(a_r))(e_i(a_s)-v_2(a_s)) \\
&=(e_i-v_1)(a) + (e_i-v_2)(a),
\end{align*}
where the third sum on the penultimate line is zero in $R\otimes \m^{i-1}/\m^i$, being a sum of products.
So we see that each successive quotient $V_i$ is a vector group as claimed.

We now recall that there is a natural copy of $G'$ inside $G_i$ -- 
on the level of $R$-points for a $k'$-algebra $R$, this comes from the observation that 
any $k'$-algebra homomorphism $k'[G'] \to R$ has a canonical
extension to a $B_i$-algebra homomorphism $A_i = k'[G']\otimes B_i \to R\otimes B_i$.
This copy of $G'$ acts on $V_i$ by conjugation, and the replacement of $v\in V_i(R)$ with 
$e_i-v$ is clearly equivariant with respect to this action. 
On the other hand, under the identification of $\Lie(G')$ with $\textrm{Der}_{k'}(k'[G'],k')$, the adjoint action comes from the action of $G'$ on $k'[G']$ induced by the conjugation action of $G'$ on itself. 
This in turn, gives rise to the action of $G_i$ on $A_i = k'[G']\otimes B_i$ induced by conjugation after Weil restriction, and the canonical copy of $G'$ acts via its action on the first factor in this tensor product.
Hence we see that the isomorphism of vector groups we have established above is $G'$-equivariant.

Now this is established, we show how to piece together the $V_i$ to give the claims in the proposition. Since the natural map $B=B_n \to B_i$ for each $1\leq i\leq n$ is the composition of the natural maps $B_n\to B_{n-1} \to \cdots \to B_i$, we also have surjective maps $G_{k'}\cong G_n \to G_i$;
let $U_i$ denote the kernel of this map for each $i$.
Then $U_1$ is the geometric unipotent radical of $G$ and $U_1\supset U_2 \supset \cdots \supset U_n$.
We also see that $U_i/U_{i+1} \cong V_{i+1}$ for each $1\leq i\leq n-1$.
So we have a filtration of the unipotent radical with successive quotients equal to the groups $V_i$.

The final step is to check that the commutator $[U_1,U_i] \subseteq U_{i+1}$ for each $i$,
so that this filtration is compatible with the lower central series for $U_1$.
For this, given a $k'$-algebra $R$, let $u\in U_1(R)$ and $v \in U_i(R)$
and as usual identify $u$ and $v$ with $B$-algebra homomorphisms $A\to R\otimes B$.
Analogously with above, we want to rewrite the commutator in a nice way.   
Recall that in this set-up we have denoted by $e_i$ the homomorphisms $A_i\to R\otimes B_i$ coming from the identity of $G_i(k')$.
To ease notation, denote $e_n=e$.
Note that for each $i$, $e_i$ coincides with $e$ modulo $\m^i$, and the homomorphisms $e$, $u$
and $v$ have the same domain and codomain.
Then one equation to use (of the many possible) is
$$
e-uvu^{-1}v^{-1} = (e-v^{-1})(e-vu^{-1}v^{-1}) - (e-u)(e-v)u^{-1}v^{-1}.
$$
Apply this to $a\in A$.
Note that since $u\in U_1$, $(e-u)(a) \in R\otimes_{k'} \m$.
Similarly, since $v\in U_i$, $(e-v) \in R\otimes_{k'} \m^{i}$ modulo $\m^{i+1}$. 
Therefore the second term on the right-hand side is zero modulo $\m^{i+1}$.
Similarly the first term is zero, noting that $vu^{-1}v^{-1} \in U_1$.
Therefore, the whole right-hand side, and hence the left-hand side, is zero in $U_i/U_{i+1}$, which is what we wanted. 
Again, this calculation can be justified using the comultiplication.
\end{proof}

When $G'$ is not unusual---that is, it is non-commutative and either $p\neq 2$ or $G'$ is not a direct product of some copies of $\SL_2$ with a torus---then we show in the next sections that the upper bound for $\cl(\RR_u(G_{\bar k}))$ in the proposition above is in fact also a lower bound.

\begin{example}\label{ex:nomofo}Note that the filtration from Proposition~\ref{lem:mofo} need not exist for an arbitrary standard pseudo-reductive group.  For example, if $p=2$, $G'=\SL_2$, $k'/k$ is a purely inseparable field extension of degree $2$ and $G$ is the standard pseudo-reductive group $\R_{k'/k}(G')/\R_{k'/k}(\mu_2)$, then $G$ has dimension $\dim G'\cdot [k':k]-(p-1) = 5$; see \cite[Ex.\ 1.3.2]{CGP15}. Since $G'=\SL_2$ is defined over $k$, there is a canonical copy of $G'$ inside $\R_{k'/k}(G')_{k'}$.  Factoring out by $\R_{k'/k}(\mu_2)_{k'}$ kills the copy of $\mu_2$ inside $G'$, so we obtain a copy of $\PGL_2$ inside $G_{k'}$.  This gives rise to a copy of $\Lie(\PGL_2)$ in $\Lie(G_{k'})$, so we see that $\Lie(\RR_u(G_{\bar k}))$ cannot be a sum of copies of $\Lie(G')_{\bar k}$, just for dimensional reasons. (Similar examples can be given for any $G'$ whose centre is not smooth.)
\end{example}

\subsection{Lower bound}

Let $k'/k$ be a purely inseparable field extension, and set $B=k'\otimes_k k'$ with maximal ideal $\m$. In this section we do a calculation in $\GL_2(B)$ which shows that $\RR_u(\R_{k'/k}(\GL_2)_{\bar k})$ has nilpotency class exactly $n-1$, where $n$ is minimal such that $\m^n=0$. (Of course the upper bound has already been established in the previous section.)

Let $G=\R_{k'/k}(\GL_2)$. It suffices to compute the nilpotency class of $\RR_{u,k'}(G_{k'})$ since this is a $k'$-form of the geometric unipotent radical. As in the proof of Proposition \ref{lem:mofo} we have $\R_{k'/k}(\GL_2)_{k'}=\R_{B/k}(\GL_2)$, and the $k'$-radical of this is the kernel of the map induced from $B\to k'$ with kernel $\m$. Thus the $k'$-points of $\RR_{u,k'}(G_{k'})$ are simply the matrices \[\begin{pmatrix}1+m_1 & m_2\\ m_3 & 1+m_4\end{pmatrix}\quad \text{for }m_i\in \m.\]
By minimality of $n$, we can find $m\in \m$ and $m_i\in \m$ for $1\leq i\leq n-2$ such that $m\cdot m_1\cdots m_{n-2}\neq 0$. Certainly it suffices to find the desired lower bound for $\cl(\RR_u(G_{\bar k}))$ on the level of the $k'$-points of $\RR_{u,k'}(G_{k'})$, so it suffices to see that there exist $z,y_1,\dots,y_{n-2}\in \RR_{u,k'}(G_{k'})(k')$ with $[y_{1},[y_{2},[\dots,[y_{n-2},z]]]]\neq 1$. In the light of that observation, take 
\[z:=\begin{pmatrix} 1 & m\\0 & 1\end{pmatrix}\quad\text{and}\quad y_i:=\begin{pmatrix}1+m_i&0\\0& 1\end{pmatrix}\ \mbox{for $1\leq i\leq n-2$}; \]
then an elementary calculation yields
\[ [y_1,[y_{2},[\dots,[y_{n-2},z]]]]=\begin{pmatrix}1 & m\cdot m_1\cdots m_{n-2}\\0 & 1\end{pmatrix}\neq 1.\]

Combining with the lower bound from Prop.~\ref{lem:mofo}, we have proved
\begin{lemma}\label{gl2case} Let $k'/k$ be a purely inseparable field extension.
Then
\[\cl(\RR_u((\R_{k'/k}(\GL_2))_{\bar k}))= n-1,\]
where $n$ is minimal such that $\m^n=0$.\end{lemma}

We wish to deduce the analogous result for $\PGL_2$ and $\SL_2$ (the latter when $p\neq 2$). To do this we need the following two lemmas, which we also use in the proof of Theorem~\ref{maintheorem}.

\begin{lemma}\label{semisimpleUnip}
Let $k'/k$ be a finite field extension and let $G'$ be a semisimple $k'$-group such that $Z':= Z_{G'}$ is smooth.  Let $G'_{\mathrm{ad}}= G'/Z'$ and let $G:=\R_{k'/k}(G')$. Then there is a natural isomorphism \[\RR_{u,k'}(G_{k'})\cong \RR_{u,k'}(\R_{k'/k}(G'_{\mathrm{ad}})_{k'}).\]\end{lemma}
\begin{proof}Since $Z'$ is smooth, there is a smooth isogeny giving rise to the exact sequence
\[1\to Z'\to G'\stackrel{\pi}{\to} G'_{\mathrm{ad}}\to 1.\]
By \cite[A.5.4(3)]{CGP15}, Weil restriction preserves the exactness of this sequence, and so there is an exact sequence
\[1\to \R_{k'/k}(Z')\to G\stackrel{\R_{k'/k}(\pi)}{\longrightarrow} \R_{k'/k}(G'_{\mathrm{ad}})\to 1.\]
This gives, after base change to $k'$, an exact sequence
\[1\to \R_{k'/k}(Z')_{k'}\to G_{k'}\stackrel{\R_{k'/k}(\pi)_{k'}}{\longrightarrow} \R_{k'/k}(G'_{\mathrm{ad}})_{k'}\to 1.\]
Now $Z'$ is a smooth finite group scheme, so $Z'\cong \R_{k'/k}(Z')_{k'}$ as algebraic groups.  This implies that $\R_{k'/k}(\pi)_{k'}$ is a smooth isogeny.  It follows that the map $\R_{k'/k}(\pi)_{\bar k}\colon G_{\bar k}\to \R_{k'/k}(G'_{\mathrm{ad}})_{\bar k}$ obtained by base change to $\bar k$ gives rise to an isomorphism from $\RR_u(G_{\bar k})$ to $\RR_u(\R_{k'/k}(G'_{\mathrm{ad}})_{\bar k})$. But $\RR_u(G_{\bar k})$ and $\RR_u(\R_{k'/k}(G'_{\mathrm{ad}})_{\bar k})$ are defined over $k'$, so $\R_{k'/k}(\pi)_{k'}$ gives an isomorphism from $\RR_{u,k'}(G_{k'})$ to $\RR_{u,k'}(\R_{k'/k}(G'_{\mathrm{ad}})_{k'})$, as required.
\end{proof}

\begin{lemma}\label{reducetoweilrest} 
Let $k',k,G', Z'$ and $G'_{\mathrm{ad}}$ be as in Lemma~\ref{semisimpleUnip}, but without the assumption that $Z'$ is smooth.  Further, let $C$ be a commutative pseudo-reductive $k$-group occurring in a factorisation of the map $\R_{k'/k}(T')\to \R_{k'/k}(T'/Z')$ for $T'$ a maximal $k'$-torus of $G'$. Then: 

(i) We have \[\cl(\RR_{u}(((\R_{k'/k}(G')\rtimes C)/\R_{k'/k}(T'))_{\bar k}))\leq \cl(\RR_{u}((\R_{k'/k}(G')\rtimes C)_{\bar k}))=\cl(\RR_{u}(\R_{k'/k}(G')_{\bar k})).\]

(ii) If moreover $Z'$ is smooth, we have
\[\cl(\RR_{u,k'}((\R_{k'/k}(G')\rtimes C)_{k'}))=\cl(\RR_{u,k'}(\R_{k'/k}(G'_{\mathrm{ad}})_{k'}))=\cl(\RR_{u,k'}(((\R_{k'/k}(G')\rtimes C)/\R_{k'/k}(T'))_{k'})),\]
where in the final term we quotient by the usual central copy of $\R_{k'/k}(T')$ occurring in the standard construction.
\end{lemma}
\begin{proof}(i) Let $H:=(\R_{k'/k}(G')\rtimes C)_{\bar k}$ and $G:=((\R_{k'/k}(G')\rtimes C)/\R_{k'/k}(T'))_{\bar k}$. We have an epimorphism $f:H\to G$. Let $U=\RR_{u}(H)$. Then $U':= f(U)$ is a smooth, connected normal unipotent subgroup of $G$, and so we get an epimorphism from the reductive group $H/U$ onto $G/U'$.  But $H/U$ is reductive, so $G/U'$ must be reductive.  This means $U'$ contains (and therefore equals) $\RR_{u}(G)$. Thus $U'$ is a quotient of $U$ and so $\cl(U')\leq \cl(U)$, proving the inequality.

For the equality, we have trivially $\cl(U)\geq \cl(\RR_{u,k'}(\R_{k'/k}(G')_{k'}))$, so we will be done if we can show $\D(U)\subseteq\D(\RR_{u}(\R_{k'/k}(G')_{\bar k}))$. Now $C$ factorises the map $\R_{k'/k}(T')\to\R_{k'/k}(T'/Z')$, whose kernel is $\R_{k'/k}(Z')$ (due to left exactness of Weil restriction). If $A$ is any $k$-algebra and $g\in C(A)$ then for any $q\in\R_{k'/k}(G')(A)$ we can find $h\in \R_{k'/k}(T')(A)$ such that $hqh^{-1}=gqg^{-1}$. In particular, $[\RR_{u,k'}(\R_{k'/k}(G'))_{k'},C_{k'}]\subseteq [\RR_{u,k'}(\R_{k'/k}(G'))_{k'},\RR_{u,k'}(\R_{k'/k}(G'))_{k'}]$ as required.

(ii)
Let $\phi:\R_{k'/k}(T')\to C$ be the map in the factorisation.  Set $Z= \R_{k'/k}(Z')$.  Thanks to the isomorphism of groups $Z'\cong Z_{k'}$ arising from the hypothesis on $Z'$, we have by the arguments of Lemma~\ref{semisimpleUnip} a smooth isogeny \[1\to Z'\times \phi_{k'}(Z')\to\R_{k'/k}(G')_{k'}\rtimes C_{k'}\to \R_{k'/k}(G'_{\mathrm{ad}})_{k'}\rtimes C_{k'}/\phi_{k'}(Z')\to 1,\]  or equivalently \[1\to Z'\times \phi_{k'}(Z')\to(\R_{k'/k}(G')\rtimes C)_{k'}\to (\R_{k'/k}(G'_{\mathrm{ad}})\rtimes D)_{k'}\to 1,\tag{$*$}\]  where we write $D$ in place of $C/\phi(Z)$. The argument of Lemma~\ref{semisimpleUnip} yields an isomorphism \[\RR_{u,k'}((\R_{k'/k}(G')\rtimes C)_{k'})\cong \RR_{u,k'}(( \R_{k'/k}(G'_{\mathrm{ad}})\rtimes D)_{k'}).\]

The map $C\to \R_{k'/k}(T'/Z')$ gives rise to a map $\kappa:D\to \R_{k'/k}(G'_{\mathrm{ad}})$.  We claim that the map $\tau:\R_{k'/k}(G'_{\mathrm{ad}})\rtimes D\to \R_{k'/k}(G'_{\mathrm{ad}})\times D$ given on the level of $A$-points by $(g,d)\mapsto (g\kappa(d),d)$ is an isomorphism. To see this, first recall $(g,d)(h,e)=(g\kappa(d)h\kappa(d)^{-1},de)$. Then $(g,d)(h,e)\mapsto (g\kappa(d)h\kappa(d)^{-1}\kappa(de),de)=(g\kappa(d)h\kappa(e),de)$ which is the product of $(g\kappa(d),d)$ and $(h\kappa(e),e)$ in the direct product as required.

With this in mind, we get \[ \RR_{u,k'}( (\R_{k'/k}(G'_{\mathrm{ad}})\rtimes D)_{k'})\cong \RR_{u,k'}( (\R_{k'/k}(G'_{\mathrm{ad}})\times D)_{k'})\cong \RR_{u,k'}(\R_{k'/k}(G'_{\mathrm{ad}})_{k'})\times \RR_{u,k'}(D_{k'}); \] but the commutativity of $D$ implies that $\cl(\RR_{u,k'}((\R_{k'/k}(G'_{\mathrm{ad}})\rtimes D))_{k'})=\cl(\RR_{u,k'}(\R_{k'/k}(G'_{\mathrm{ad}})_{k'}))$.  This proves the first equality of the lemma.

To see the second equality, observe that $\R_{k'/k}(T'/Z')\cong \R_{k'/k}(T')/Z$ since $Z'$ is smooth.  We see that the map $(\R_{k'/k}(G')\rtimes C)_{k'}\to (\R_{k'/k}(G'_{\mathrm{ad}})\rtimes D)_{k'}$ from ($*$) takes the copy of $\R_{k'/k}(T')_{k'}$ onto the copy of $\R_{k'/k}(T'/Z')_{k'}$, so the induced map of quotient groups yields an isogeny
\[1\to N\to(\R_{k'/k}(G')\rtimes C)_{k'}/\R_{k'/k}(T')_{k'}\to (\R_{k'/k}(G'_{\mathrm{ad}})\rtimes D)_{k'}/\R_{k'/k}(T'/Z')_{k'}\to 1 \]
for some $N$.  This isogeny is smooth because the canonical projections to the quotient groups are smooth.  By the argument of Lemma~\ref{semisimpleUnip} again, we get an isomorphism
\[ \RR_{u,k'}((\R_{k'/k}(G')\rtimes C)_{k'}/\R_{k'/k}(T')_{k'})\cong \RR_{u,k'}((\R_{k'/k}(G'_{\mathrm{ad}})\rtimes D)_{k'}/\R_{k'/k}(T'/Z')_{k'}). \]
Now $\phi_{k'}$ induces a map $\Phi$ from $\R_{k'/k}(T'/Z')_{k'}$ onto a smooth subgroup of $D$, and it is easily checked that $\tau$ gives rise to an isomorphism
\[
(\R_{k'/k}(G'_{\mathrm{ad}})\rtimes D)/\R_{k'/k}(T'/Z')\cong \R_{k'/k}(G'_{\mathrm{ad}})\times D/\Phi(\R_{k'/k}(T'/Z')).
\] 
The commutativity of $D$ again gives the result.
\end{proof}

We can now tackle
\begin{prop}\label{prop:casegl2orpgl2} Suppose $k'/k$ is a purely inseparable field extension and $G=\R_{k'/k}(G')$, where $G'$ is one of $\GL_2$ or $\PGL_2$ or $\SL_2$ (the final case only if $\Char k\neq 2$). Let $n$ be minimal such that $\m^n=0$. Then $\cl(\RR_{u}(G_{\bar k}))=n-1$.\end{prop}
\begin{proof}
In the cases considered $Z_{G'}$ is smooth, so by Lemmas~\ref{semisimpleUnip} and~\ref{reducetoweilrest} the nilpotency classes of $\RR_{u,k'}(\R_{k'/k}(\GL_2)_{k'})$, $\RR_{u,k'}(\R_{k'/k}(\PGL_2)_{k'})$ and $\RR_{u,k'}(\R_{k'/k}(\SL_2)_{k'})$ are all equal.  But Lemma \ref{gl2case} now supplies the result.
\end{proof}

\subsection{$\SL_2$ in characteristic 2}In this section we wish to study the case  $G'=\SL_2$ and $G=\R_{k'/k}(G')$, where $\Char k=2$ and $k'/k$ is a purely inseparable field extension.

We will show

\begin{lemma}\label{lem:sl2p2lem}
 Let $k'/k$ be a purely inseparable finite field extension and let $B=k'\otimes_k k'$, with maximal ideal $\m$. Furthermore, denote by ${}^2\m$ the ideal generated by the squares of the elements in $\m$. Suppose $n$ is the minimum such that $({}^{2}\m)^{n-1}\cdot\m^2=0$. Then $\cl(\RR_u(G_{\bar k}))= n$.\end{lemma}

\begin{proof}
 By Lemma \ref{reducetosplit} we may assume $k=k_s$.  We use arguments analogous to those in the previous section.
First we establish the upper bound.
Recall that we have $R:=\R_{u,k'}(G_{k'})=\ker \varphi$ where $\varphi:G_{k'}\to G'$ is the map induced by reduction modulo $\m$. Since we will be working with commutator subgroups it is convenient to work with the $\bar k$-points of $R$.  Set $\barB= B\otimes_{k'} {\bar k}= k'\otimes_k \bar k$ and $\barm= \m \otimes_{k'} {\bar k}$.  Then $\barB$ is local with unique maximal ideal $\barm$ and quotient field $\bar k$.  The ideal $\barm$ is nilpotent and the minimal $n$ with $\barm^n = 0$ is the same as the minimal $n$ with $\m^n = 0$.

Define ideals $I_r$ and $J_r$ of $B$ for $r\geq 0$ by $I_0= J_0= \m$, $I_r= ({}^{2}\m)^r\cdot\m$ for $r\geq 1$ and $J_r= ({}^{2}\m)^{r-1}\cdot\m^2$ for $r\geq 1$.  Set ${\bar I}_r= I_r\otimes_{k'} {\bar k}= ({}^{2}\barm)^r\cdot\barm$ and ${\bar J}_r= J_r\otimes_{k'} {\bar k}= ({}^{2}\barm)^{r-1}\cdot\barm^2$.  It is clear that ${\bar J}_r= 0$ (resp., ${\bar I}_r= 0$) if and only if $J_r= 0$ (resp., $I_r=0$) for each $r$.  Observe that ${\bar I}_r\subseteq {\bar J}_r$, ${}^2\barm\cdot {\bar I}_r\subseteq {\bar I}_{r+1}$, ${}^2\barm\cdot {\bar J}_r\subseteq {\bar J}_{r+1}$ and $\barm\cdot {\bar I}_r\subseteq {\bar J}_{r+1}$ for all $r$.

Define $R_r\subseteq G(\bar k)$ by
\[R_r=\left\{\left.\begin{pmatrix}1+m_1 & m_2\\ m_3 & 1+m_4\end{pmatrix}\ \ \right| \quad m_1,m_4\in {\bar J}_r, m_2,m_3\in {\bar I}_r,\ m_1+m_4+m_1m_4+m_2m_3=0\right\}.\]
It is easily checked that $R_0= R({\bar k})$ and each $R_r$ is a subgroup of $R({\bar k})$.  Let $M= \begin{pmatrix}1+m_1 & m_2\\ m_3 & 1+m_4\end{pmatrix}\in R_r$ for some given $r$.  Consider the following matrices of $R_0$:
\begin{align*}
M_1&:=\begin{pmatrix}1 & m\\ 0 & 1\end{pmatrix},\qquad\text{ with }m\in\barm;\\
M_2&:=\begin{pmatrix}1+m & 0\\ 0 & (1+m)^{-1}\end{pmatrix},\qquad\text{ with }m\in\barm.\end{align*}
We have $M^{-1}=\begin{pmatrix}1+m_4 & m_2\\ m_3 & 1+m_1\end{pmatrix}$ and so we find the commutator 
\[[M_1,M]=\begin{pmatrix}1+mm_3+mm_1m_3+m^2m_3^2 & mm_1^2+m^2m_3+m^2m_1m_3\\ mm_3^2 & 1+mm_3+mm_1m_3\end{pmatrix}.\]
We see that $[M_1,M]$ belongs to $R_{r+1}$.  By symmetry, the commutator $[M_1^\top,M]$ also belongs to $R_{r+1}$ (here $M_1^\top$ denotes the transpose of $M_1$). Similarly, noting that $(1+m)^{-1}=1+m+m'$ for some $m'\in{}^2\barm$, we have
\[[M_2,M]=\begin{pmatrix}1+m^2m_2m_3 & m^2(1+m_1)m_2\\
(m^2+m'^2)m_3(1+m_4) & 1+(m^2+ (m')^2)m_2m_3\end{pmatrix},\] which belongs to $R_{r+1}$. 
These calculations together show that if $n$ is such that ${\bar J}_n=0$, then
$[X_1,[X_2,\dots,[X_n,g]]]=1$ for arbitrary $g\in R(\bar{k})$ and arbitrary $X_i$ chosen from elements of the form $M_1$, $M_1^\top$ and $M_2$.

Let $(U')^-$ denote the lower root group of $G'$, $T'$ denote the diagonal torus and $U'$ denote the upper root group, and set $U:=\R_{k'/k}(U')$, $C:=\R_{k'/k}(T')$, $U^-:=\R_{k'/k}((U')^-)$. Then the multiplication map $\mu:U^-\times C\times U\to G$ is an open immersion by \cite[Cor.~3.3.16]{CGP15}.
Moreover, the intersections $\R_{k'/k}(U)_{\bar k}\cap R_{\bar k}$, $\R_{k'/k}(T)_{\bar k}\cap R_{\bar k}$, $\R_{k'/k}(U^-)_{\bar k}\cap R_{\bar k}$ have codimension $1$ in $\R_{k'/k}(U)_{\bar k}$, $\R_{k'/k}(T)_{\bar k}$, $\R_{k'/k}(U^-)_{\bar k}$, respectively; this can be seen by looking at the kernel of the restriction of $\varphi_{\bar k}$ to these subgroups, where $\varphi$ is the map above.  It follows that $\mu_{\bar k}(D)$ has codimension 3 in $\R_{k'/k}(G')_{\bar k}$, where $D:= (\R_{k'/k}(U)_{\bar k}\cap R_{\bar k})\times (\R_{k'/k}(T)_{\bar k}\cap R_{\bar k})\times (\R_{k'/k}(U^-)_{\bar k}\cap R_{\bar k})$.  In particular, $\mu_{\bar k}(D)$ is an open subset of $R_{\bar k}$, hence generates $R_{\bar k}$.
But, in light of the identity $[x,zy]=[x,y]\cdot[x,z]^y$ and the fact from the previous paragraph that $[X_1,[X_2,\dots,[X_n,R(\bar k)]]]=1$ for generators $X_i$ of $R(\bar k)$, we are done for the upper bound.

For the lower bound, we exhibit an explicit nonzero commutator, in an analogous way to  Lemma~\ref{gl2case}. Let $n$ be as in the statement.
If $n=1$ then there is nothing to do, so suppose $n>1$.
Since we're assuming that $n$ is minimal such that $J_n= 0$, we may pick $m_1,\dots,m_{n-2}\in \m$ and $m,m'\in \m$ such that $m_1^2\cdots m_{n-2}^2mm'$ is non-zero. Take \[z=\begin{pmatrix}1 & m\\0&1\end{pmatrix},\qquad z'=\begin{pmatrix}1 & 0\\m'&1\end{pmatrix},\qquad y_i=\begin{pmatrix}1+m_i&0\\0&(1+m_i)^{-1}\end{pmatrix}\]
for $1\leq i\leq n-2$.  Set $w= [y_1,[y_2,[\dots,[y_{n-2},z]]]]$.  Then a calculation yields
\[ w=\begin{pmatrix}1& m_1^2\cdots m_{n-2}^2m\\0&1\end{pmatrix} \]
and
\[ [z',w]= \begin{pmatrix}1+ m_1^2\cdots m_{n-2}^2mm'&0\\0&1+ m_1^2\cdots m_{n-2}^2mm'\end{pmatrix}\neq 1. \]
This gives the lower bound.
\end{proof}

\begin{remark}
\label{rem:weirdquot}
 Consider the pseudo-reductive group $\R_{k'/k}(\SL_2)/\R_{k'/k}(\mu_2)$.  Since the canonical projection $\R_{k'/k}(\SL_2)\to \R_{k'/k}(\SL_2)/\R_{k'/k}(\mu_2)$ is surjective on ${\bar k}$-points, $\RR_u((\R_{k'/k}(\SL_2)/\R_{k'/k}(\mu_2))_{\bar k})({\bar k})$ is a quotient of $\RR_u(\R_{k'/k}(\SL_2)_{\bar k})({\bar k})$, so the nilpotency class $c$ of the latter is at most $n$ (where $n$ is as in Lemma~\ref{lem:sl2p2lem}).  The image of the commutator $w$ from above in $\RR_u((\R_{k'/k}(\SL_2)/\R_{k'/k}(\mu_2))_{\bar k})({\bar k})$ does not vanish, so $c$ is at least $n-1$ (note that we may identify $\RR_u(\R_{k'/k}(\mu_2)_{\bar k})({\bar k})$ with the group of matrices of the form $\begin{pmatrix}1+m&0\\0&1+ m\end{pmatrix}$ with $m\in \barm$ and $m^2=0$).  The image of $[z',w]$, however, does vanish.  It follows that $c$ is either $n$ or $n-1$; we do not know which.
\end{remark}

\begin{example}
(i). Consider the case that $k' = k(t)$ with $t^2\in k\setminus k^2$.
Then $\m^2 = 0$, so $J_1= 0$.  We can calculate:
$$
\begin{pmatrix}1+m_1 & m_2\\ m_3 & 1+m_4\end{pmatrix}\begin{pmatrix}1+m_1' & m_2'\\ m_3' & 1+m_4'\end{pmatrix}
= 
\begin{pmatrix}1+m_1+m_1'  & m_2+m_2'\\ m_3+m_3' & 1+m_4+m_4'\end{pmatrix}
$$
for all $m_i \in \barm$.
Thus the unipotent radical is abelian in this case, so the nilpotency class is indeed $1$.

(ii). Now let $k' = k(s,t)$ of degree $4$, where $s^2,t^2\in k\setminus k^2$.  We see that ${}^2\m = 0$ but $\m^2\neq 0$, so $J_1\neq 0$ and $J_2= 0$.  Choose $m_1=s\otimes 1 +1\otimes s \in \m$ and $m_2 = t\otimes 1 + 1\otimes t\in \m$.
Note that $m_1m_2\neq 0$, and hence if we set
\begin{align*}
u_1&:= \begin{pmatrix}1 & m_1 \\ 0 & 1\end{pmatrix},\\
u_2&:= \begin{pmatrix}1 & 0 \\ m_2 & 1\end{pmatrix},
\end{align*}
we have $u_1u_2\neq u_2u_1$, so the unipotent radical in this case is not abelian.  In fact, we see from Lemma~\ref{lem:sl2p2lem} that the nilpotency class is $2$.
\end{example}

\subsection{The case of Weil restrictions}We will treat here the case that $G$ is a Weil restriction $\R_{k'/k}(G')$ for $G'$ a reductive $k'$-group, where $k'$ is a finite reduced non-zero $k$-algebra. We deal with some special situations relating to the case where some fibre of $G'$ is unusual.

\begin{lemma}
\label{lem:GL2_crit}
 Suppose $p= 2$.  Let $H'= \SL_2$ and suppose $G'= S'H'$ is a split reductive group, where $S'$ is a central torus of $G'$.  Then either the product map $S'\times H'\to G'$ is an isomorphism or $G'$ contains a copy of $\GL_2$.
\end{lemma}

\begin{proof}
 The group $G'$ has semisimple rank 1, so by \cite[Thm.\ 21.55]{milneAGS}, $G'$ is isomorphic to one of $\SL_2\times T'$, $\GL_2\times T'$ or $\PGL_2\times T'$ for some torus $T'$.  The result follows.
\end{proof}

\begin{lemma}
\label{lem:no_GL2_or_PGL_2}
Let $k'$ be a field and $G'$ a noncommutative split reductive $k'$-group. Then exactly one of the following holds: 

(i) $G'$ has a subgroup isomorphic to $\GL_2$ or $\PGL_2$;

(ii) $G'=(\SL_2)^r\times S'$ for some $r\geq 1$.
\end{lemma}

\begin{proof}
Suppose $G'$ has no subgroup isomorphic to $\GL_2$ or $\PGL_2$. If $G'$ contains a subgroup of type $A_2$ then $G'$ contains a copy of $\GL_2$, which is impossible; note that this also implies $G'$ does not contain a subgroup of type $G_2$.  If $G'$ has a simple factor $H'$ of type $B_2$ then the long-root Levi subgroup of $H'$ is isomorphic to $\GL_2$ if $H'$ is simply connected, and to ${\mathbb G}_m\times \PGL_2$ if $H'$ is adjoint, which gives a contradiction in both cases.  So every simple factor of $G'$ is isomorphic to $\SL_2$.
 
Let $H_1',\ldots, H_r'$ be the simple factors of $G'$ and let $S'= Z(G')^0$.  Multiplication gives a central isogeny $\varphi\colon H_1'\times\cdots \times H_r'\times S'\to G'$. 
Suppose $\varphi$ is not an isomorphism.  Then the projection of $\ker(\varphi)$ to $H_1'$, say, must be non-trivial. Let $T'$ be a split maximal torus of $H_2'\times\dots\times H_r'\times S'$. Since $\ker(\varphi)$ is contained in a maximal torus of $H_1'\times T'$, we must have that $\varphi$ is not injective on $H_1'\times T'$. It follows from Lemma~\ref{lem:GL2_crit} that $H_1'T'\subseteq G'$ contains a copy of $\GL_2$.
\end{proof}

We can now prove Theorem \ref{maintheorem} for the case of Weil restrictions of reductive groups. Note that such groups are of course standard pseudo-reductive, taking $C=\R_{k'/k}(T')$, and $\phi$ an isomorphism.

\begin{prop}\label{prop:Weil_class} Theorem \ref{maintheorem} holds when $G=\R_{k'/k}(G')$.\end{prop}
\begin{proof}
By Lemma \ref{reducetosplit} we may assume $k=k_s$.
Since $G$ is a direct product of $\R_{k_i/k}(G_i')$ for some inseparable field extensions $k_i/k$ and fibres $G_i'$ of $G'$ over $k'$, we will clearly have $\cl(\RR_u(G_{\bar k}))=\max_i(\RR_u(\R_{k_i/k}(G_i')_{\bar k}))$. Hence it suffices to show that the nilpotency class of the unipotent radical of $G_i:=\R_{k_i/k}(G_i')$ is the number $\ell_i$ from the introduction.
The case that $G_i'$ is commutative is obvious, so assume $G_i'$ is non-commutative. It is of course split since $k=k_s$. If $G_i'$ is not unusual, then Lemma \ref{lem:no_GL2_or_PGL_2} says that $G_i'$ contains a copy of $\PGL_2$ or $\GL_2$. In this case Proposition \ref{prop:casegl2orpgl2} gives $\cl(\RR_u((G_i)_{\bar k}))\geq \ell_i$ and Proposition \ref{lem:mofo} gives $\cl(\RR_u((G_i)_{\bar k}))\leq \ell_i$. Lastly, if $G_i'$ is unusual, then $G_i$ is isomorphic to $\R_{k_i/k}(\SL_2)^r\times \R_{k_i/k}(S')$ for some torus $S'$ and some $r\geq 1$, by Lemma~\ref{lem:no_GL2_or_PGL_2}. Then the result follows from Lemma \ref{lem:sl2p2lem}.
\end{proof}

\subsection{Centrally smooth enlargements}
\label{subsec:hatty}

The following is essentially a definition from \cite{Mar99}.

\begin{defn}
 Let $G'$ be a reductive $k'$-group, where $k'$ is a field.  A {\em centrally smooth enlargement} of $G'$ is a reductive $k'$-group $\widehat{G'}$ such that $Z(\widehat{G'})$ is smooth, $G'\subseteq \widehat{G'}$ and $\widehat{G'}$ is generated by $G'$ and a central torus.
\end{defn}

\begin{lemma}
\label{lem:hatty}
 Let $G'$ be a split reductive $k'$-group.  There exists a centrally smooth  enlargement $\widehat{G'}$ of $G'$.
\end{lemma}
\begin{proof}
 This follows from \cite[Thm.\ 4.5]{Mar99}.  The result in {\em loc.\ cit.} is stated only for algebraically closed fields, but the proof works for arbitrary split reductive groups, by the existence and isogeny theorems for split reductive groups \cite[16.3.2,\ 16.3.3]{Spr98}.
\end{proof}

\begin{remark} There is also another, simpler, proof which works for the non-split case as well: embed $Z(G')$ in a torus $S'$: for example, we could choose $S'$ to be a maximal torus of $G'$.  Then set $\widehat{G'}= (G'\times S')/N$, where $N$ is the image of $Z(G')$ under the obvious diagonal embedding. (Since $N\cap (1\times S')=1$, $S'$ appears as a subgroup of $\widehat{G'}$.) \end{remark}

Let $(G', k'/k, T', C)$ be a 4-tuple corresponding to a standard pseudo-reductive group \[G= (\R_{k'/k}(G')\rtimes C)/\R_{k'/k}(T'),\] with a factorisation $\R_{k'/k}(T')\stackrel{\phi}{\to}C\stackrel{\psi}{\to}\R_{k'/k}(T'/Z_{G'})$.  Assume $k'$ is a field.  Suppose $\widehat{G'}= G'S'$ is a central enlargement of $G'$, where $S'$ is a central torus of $\widehat{G'}$. 
 Set $Z'= Z(G')$, $\widehat{Z'}= Z(\widehat{G'})$ and $\widehat{T'}= T'S'$. It is easy to see that $\widehat{Z'}=Z'S'\subseteq G'S'$. Therefore the composition $T'\to \widehat{T'}\to \widehat{T'}/\widehat{Z'}$ induces an isomorphism $\beta\colon T'/Z'\to \widehat{T'}/\widehat{Z'}$. 

\begin{lemma}
\label{lem:hatty_bound}
 $\cl(\RR_{u,k'}(G_{k'}))\leq \cl(\RR_{u,k'}(\R_{k'/k}(\widehat{G'})_{k'}))$.
\end{lemma}

\begin{proof}
The idea is to embed $\R_{k'/k}(G')\rtimes C$ in a group of the form $\R_{k'/k}(\widehat{G'})\rtimes \widehat{C}$ arising from a 4-tuple $(\widehat{G'}, k'/k, \widehat{T'}, \widehat{C})$ and a factorisation $\R_{k'/k}(\widehat{T'})\stackrel{\widehat{\phi}}{\to}\widehat{C}\stackrel{\widehat{\psi}}{\to}\R_{k'/k}(\widehat{T'}/\widehat{Z'})$; the latter group is easier to understand than the former because $\widehat{Z'}$ is smooth.  Define a homomorphism $\kappa\colon \R_{k'/k}(T')\to C\times \R_{k'/k}(\widehat{T'})$ by $\kappa= \phi\times \iota^{-1}$, where $\iota\colon \R_{k'/k}(\widehat{T'})\to \R_{k'/k}(\widehat{T'})$ is inversion and we identify $\R_{k'/k}(T')$ naturally as a subgroup of $\R_{k'/k}(\widehat{T'})$, say via a map $i_{T'}$. 
Note that $\kappa$ maps $\R_{k'/k}(T')$ isomorphically onto its image, which we denote by $N$. 
Set $\widehat{C}= (C\times \R_{k'/k}(\widehat{T'}))/N$ and let $\pi_{\widehat C}\colon C\times \R_{k'/k}(\widehat{T'})\to \widehat{C}$ be the canonical projection.  
We claim that $\widehat{C}$ is pseudo-reductive.
To see this, let $V$ denote the preimage of $\RR_{u,k}(\widehat{C})$ in $C\times \R_{k'/k}(\widehat{T'})$.
Projection onto the second factor here maps $N$ isomorphically onto the natural copy of $\R_{k'/k}(T')$ in the second factor.
Hence if we project to the second factor and then quotient by this copy of $\R_{k'/k}(T')$, the subgroup $V$ gets mapped to a smooth unipotent $k$-subgroup of the commutative group $\R_{k'/k}(\widehat{T'})/\R_{k'/k}(T')$.
Since this last group is a subgroup of the commutative pseudo-reductive group $\R_{k'/k}(\widehat{T'}/T')$, it is pseudo-reductive,
and hence we conclude that the projection of $V$ to the second factor of $C\times \R_{k'/k}(\widehat{T'})$
is contained in the projection of $N$.
Hence $V\subseteq C\times \R_{k'/k}(T')$.

We have a  map $\psi:C\times \R_{k'/k}(T') \to C$ defined by $\psi(c,x) = c\phi(x)^{-1}$ for $c\in C(A)$ and $x\in \R_{k'/k}(T')(A)$, where $A$ is any $k$-algebra.  Clearly the kernel of $\psi$ is $N$, which is smooth, so $\psi$ is smooth.  Hence $\psi(V)$ is a smooth unipotent $k$-subgroup of $C$, and is therefore trivial.  It follows that $V\subseteq N$, so $\RR_{u,k}(\widehat{C})=1$ as required.

We have obvious inclusions $i_C\colon C\to C\times \R_{k'/k}(\widehat{T'})$ and $i'\colon \R_{k'/k}(\widehat{T'})\to C\times \R_{k'/k}(\widehat{T'})$.  Define $j\colon C\to \widehat{C}$ by $j= \pi_{\widehat C}\circ i_C$; then $j$ is an embedding of $C$ in $\widehat{C}$ since $N\cap (C\times 1)= 1$. Now define $\widehat{\phi}\colon \R_{k'/k}(\widehat{T'})\to \widehat{C}$ by $\widehat{\phi}= \pi_{\widehat C}\circ i'$.  It follows from the construction that $j\circ \phi= \widehat{\phi}\circ \R_{k'/k}(i_{T'})$. 

 Consider the homomorphism $C\times \R_{k'/k}(\widehat{T'})\to \R_{k'/k}(\widehat{T'}/\widehat{Z'})$ given by the composition
 $$ C\times \R_{k'/k}(\widehat{T'})\stackrel{\psi\times{\rm id}}{\longrightarrow} \R_{k'/k}(T'/Z')\times \R_{k'/k}(\widehat{T'}) \stackrel{\R_{k'/k}(\beta)\times \R_{k'/k}(\pi_{\widehat{T'}})}{\longrightarrow}\R_{k'/k}(\widehat{T'}/\widehat{Z'})\times \R_{k'/k}(\widehat{T'}/\widehat{Z'})\to \R_{k'/k}(\widehat{T'}/\widehat{Z'}), $$
 where the third map is multiplication and $\pi_{\widehat{T'}}\colon \widehat{T'}\to \widehat{T'}/\widehat{Z'}$ is the canonical projection.  The kernel of this homomorphism contains $N$, so we get an induced homomorphism $\widehat{\psi}\colon \widehat{C}\to \R_{k'/k}(\widehat{T'}/\widehat{Z'})$.  It follows from the construction that $\widehat{\psi}\circ j= \R_{k'/k}(\beta)\circ \psi$.  One can now check that the map $\R_{k'/k}(i_{G'})\times j\colon \R_{k'/k}(G')\rtimes C\to \R_{k'/k}(\widehat{G'})\rtimes \widehat{C}$ is an embedding of groups, where $i_{G'}$ is the inclusion of $G'$ in $\widehat{G'}$.
 
 Now $G$ is a quotient of $\R_{k'/k}(G')\rtimes C$, so
 $$ \cl(\RR_{u,k'}(G_{k'}))\leq \cl(\RR_{u,k'}((\R_{k'/k}(G')\rtimes C)_{k'}))\leq \cl(\RR_{u,k'}((\R_{k'/k}(\widehat{G'})\rtimes \widehat{C})_{k'})) $$
 $$ = \cl(\RR_{u,k'}(\R_{k'/k}(\widehat{G'}/\widehat{Z'})_{k'}))\leq \cl(\RR_{u,k'}(\R_{k'/k}(\widehat{G'})_{k'})); $$
 here the equality follows from Lemma~\ref{reducetoweilrest}(ii) applied to $(\widehat{G'}, k'/k, \widehat{T'}, \widehat{C})$, and the final inequality holds because the canonical projection $\widehat{G'}\to \widehat{G'}/\widehat{Z'}$ is smooth.  This proves the result.
 \end{proof}

\begin{proof}[Proof of Theorem \ref{maintheorem}] As in the proof of Proposition \ref{prop:Weil_class} we may reduce to the case that $k=k_s$ and by passing to a fibre of $G'$ above $k'$, we may assume $k'/k$ is a purely inseparable field extension. Certainly we are done if $G'$ is commutative, so assume otherwise. 

If $G'$ is unusual, then by assumption, the map $\R_{k'/k}(T')\to C$ is a monomorphism. 
As in the proof of Lemma \ref{reducetoweilrest} we have that the semidirect product $\R_{k'/k}(G')\rtimes C$ is actually isomorphic to a direct product via $(g,d)\mapsto (gd,d)$, and so the quotient by the central subgroup $N:=\R_{k'/k}(T')\subseteq \R_{k'/k}(G')\rtimes C$ from the standard construction gives the group $\R_{k'/k}(G')\times C'$, where $C'\cong C/\R_{k'/k}(T')$. In particular, by Proposition~\ref{prop:Weil_class} the nilpotency class of the unipotent radical is as claimed. Hence we can assume $G'$ is not unusual.

We get an upper bound for $\cl(\RR_u(\R_{k'/k}(G')_{\bar k}))$ from Lemma~\ref{lem:hatty_bound}. Proposition~\ref{prop:Weil_class} applied to $\widehat{G'}$, where $\widehat{G'}$ is a central enlargement of $G'$, then gives the desired upper bound. 

To complete the proof, it is enough to show that $\cl(\RR_u(\R_{k'/k}(G')_{k'}))\leq \cl(\RR_u(G_{k'}))$. 

By Lemma \ref{lem:no_GL2_or_PGL_2} and Proposition~\ref{prop:Weil_class}, $G'$ has a subgroup $H'$ isomorphic to one of $\SL_2$, $\GL_2$ or $\PGL_2$ such that \[\cl(\RR_u(\R_{k'/k}(G')_{k'}))= \cl(\RR_u(\R_{k'/k}(H')_{k'})).\] Furthermore, if $p=2$, since $G'$ is not unusual, it follows that $H$ is $\GL_2$ or $\PGL_2$. We may regard $\R_{k'/k}(H')\subseteq \R_{k'/k}(G')$ as subgroups of $\R_{k'/k}(G')\rtimes C$ in the obvious way.  Let $\pi\colon \R_{k'/k}(G')\rtimes C\to G$ be the projection of the standard construction.  Now $\ker(\pi)$ is a smooth central subgroup of $\R_{k'/k}(G')\rtimes C$ and it follows from the definition of $\pi$ that $\ker(\pi)\cap \R_{k'/k}(G')$ is contained in the kernel of the map $\R_{k'/k}(T')\to \R_{k'/k}(T'/Z')$, which is $\R_{k'/k}(Z')$ by \cite[Prop.\ 1.3.4]{CGP15}.  So $M:= \ker(\pi)\cap \R_{k'/k}(H')$ is a central subgroup of $\R_{k'/k}(H')$.  Now $\R_{k'/k}(H')/M$ is a subgroup of $G$, so $\RR_u((\R_{k'/k}(H')/M)_{\bar k})$ is a subgroup of $\RR_u(G_{\bar k})$.  But if $Z_{H'}$ denotes the centre of $H'$ then
$$ \cl(\RR_u((\R_{k'/k}(H')/M)_{\bar k}))\geq \cl(\RR_u((\R_{k'/k}(H')/\R_{k'/k}(Z_{H'}))_{\bar k}))= \cl(\RR_u(\R_{k'/k}(H'/Z_{H'})_{\bar k})) $$ $$ =\cl(\RR_u(\R_{k'/k}(\PGL_2)_{\bar k})) $$
since $Z_{H'}$ is smooth, and we are done by Proposition~\ref{prop:Weil_class}.
\end{proof}

\subsection{When $k'/k$ is a primitive field extension.} Here we work out the consequences of Theorem \ref{maintheorem} in case $k'=k(t)$ for $q=p^e$ the order of $t$ in $k'/k$. As ever, we may assume $k$ is separably closed so that $k'/k$ is purely inseparable. In this situation, the ideal $\m=\la 1\otimes t - t\otimes 1\ra$ of $B$ is a principal ideal and it is easy to see that the minimal $n$ such that $\m^n=0$ is just $q$, since $(1\otimes t - t\otimes 1)^q=1\otimes t^q-t^q\otimes 1=t^q\otimes 1-t^q\otimes 1=0$. Hence the value of $N$ in Theorem \ref{maintheorem} is $q-1$. In particular, if $G'$ is not unusual or commutative, then $\cl(\R_{k'/k}(G'))=q-1$.

On the other hand, if $G'$ is unusual, then $p=2$ and ${}^{2}\m=\m^2$. From this one readily deduces that the value of $N$ in Theorem \ref{maintheorem} is just $q/2$ and in particular $\cl(\RR_u(\R_{k'/k}(G')_{\bar k}))=q/2$.

\begin{remark}
\label{rem:badstd}
Applying the standard construction to the Weil restriction of a reductive group can change the nilpotency class of the unipotent radical; in other words, it can happen that $\cl(\RR_u(\R_{k'/k}(G')_{\bar k}))$ is strictly less than $\cl(\RR_u(G_{\bar k}))$.  For instance, suppose $p= 2$ and $k'/k$ is a primitive purely inseparable field extension of degree $q\geq 4$.  Let $G'= \SL_2$, let $T'$ be a maximal torus of $G'$ and let $Z'= Z(G')$.  Let $C= \R_{k'/k}(T'/Z')$ and take the factorisation $\R_{k'/k}(T')\to C\to R_{k'/k}(T'/Z')$ to be the obvious one.  Then $G= (\R_{k'/k}(G')\rtimes C)/\R_{k'/k}(T')\cong \R_{k'/k}(\PGL_2)$, so $\cl(\RR_u(G_{k'}))= q-1$; but $\cl(\RR_u(\R_{k'/k}(G')_{k'}))= \cl(\RR_u(\R_{k'/k}(\SL_2)_{k'}))= \frac{q}{2}$ by the remarks above.

Note, however, that by Lemma~\ref{lem:hatty_bound}---which holds for an arbitrary standard pseudo-reductive group---the nilpotency class of $\RR_u(G_{\bar k})$ is bounded above by $n$, where $n$ is the smallest positive integer such that $\m^n$ vanishes.  This is the case even when $k'/k$ is not primitive.
\end{remark}

\section{Orders of elements in unipotent radicals of Weil restrictions}\label{sec:orders}

In this section we collect some observations about the orders of elements in unipotent radicals $\RR_u(G_{\bar k})(\bar k)$,
where $G = \R_{k'/k}(G')$ is the Weil restriction of a reductive group $G'$. More precisely, we wish to say something about the exponent $e$ of $G$; that is, the minimal integer such that the $p^e$-power map on $G$ factors through the trivial group. Clearly $e$ is stable under base change, so by Lemma \ref{reducetosplit} it suffices to consider the case where $k$ is separably closed. Of course one could let $k'$ be a non-zero reduced $k$-algebra, but by passing to a fibre, it does no harm to assume $k'/k$ is a purely inseparable field extension. 

Since $k'/k$ is purely inseparable, Proposition \ref{lem:mofo} provides a crude upper bound:
as usual, letting $B=k'\otimes_k k'$ with maximal ideal $\m$, we have a filtration of the unipotent radical $\RR_u(G_{\bar k})$ with $n$ abelian quotients, where $n$ is the minimal positive integer such that $\m^n = 0$.
Since we know each quotient is a vector group isomorphic to a direct sum of copies of the adjoint module for $G'$,
all elements in the quotients have order $p$, and we see that $p^n$ is an upper bound for the order of elements\
in $\RR_u(G)$.
Our next result shows that we can immediately do better than this with quite an elementary argument.

\begin{lemma}\label{lem:nbound}
With notation as above, let $s$ be minimal such that $p^s\geq n$. 
Then the exponent of $\RR_u(G)$ is at most $s$.
\end{lemma}

\begin{proof}
Let $\barB$ and $\barm$ be as in Lemma~\ref{lem:sl2p2lem}.  Since $G$ is smooth, it is enough to consider elements of $G(\bar k)$, which we can identify with $G'(\barB)$.  By embedding $G'$ into $\GL_r$ for some $r$,  
we may assume that the elements of interest are matrices with entries in $\barB$, 
and we may identify points of the unipotent radical
as the kernel of the map $G'(\barB) \to G'(\bar k)$ given by reduction modulo $\barm$.
Then a typical element of the unipotent radical has the form $I+M$, where $I$ is the $r\times r$ identity matrix and $M$ is an $r\times r$ matrix with entries in $\barm$. 
We are therefore after a $p^{\rm th}$ power $p^s$ which kills all such matrices $M$---for then $(I+M)^{p^s} = I + M^{p^s} = I$ in all cases.

Now note that the entries of a power $M^{p^s}$ are homogeneous polynomials in the entries of $M$ of degree $p^s$.
Since $\barm^n = 0$, all such polynomials vanish as soon as $p^s$ is at least $n$.
Hence, if we choose $s$ as in the statement, we are done.
\end{proof}

The following example shows that the bound provided by Lemma \ref{lem:nbound} is sharp when the rank of $G'$ is large enough.

\begin{example}
Suppose $G' = \GL_n$ where $n$ is as above.
Then we can find elements $m_1,\ldots,m_{n-1} \in \barm$ such that $m_1\cdots m_{n-1} \neq 0$,
and the element
$$
x = \left(\begin{array}{ccccc}
0 & m_1 & 0 & \cdots & 0 \\
0 & 0 & m_2 & \cdots & 0 \\
\vdots & \vdots & \ddots & \ddots & \vdots \\
0 & 0 & 0 & \cdots &  m_{n-1} \\
0 & 0 & 0 & \cdots & 0 \\
\end{array}\right)
$$ 
has $x^{n-1} \neq 0$, so the bound in Lemma \ref{lem:nbound} cannot be improved in this case.
\end{example}

In contrast to the above example, we can see in some easy examples that the bound provided by Lemma \ref{lem:nbound} can be much 
too large.

\begin{example}\label{ex:rankbound}
(i). Suppose $G = \R_{k'/k}(\GL_2)$ when $p=2$ and $k'/k$ is a purely inseparable field extension of exponent $1$
(recall that the \emph{exponent} of such a purely inseparable extension is the minimal $e$ such that $x^{p^e} \in k$ for all $x\in k'$). 
The fact that $k'/k$ has exponent $1$ means in this case that $m^2 = 0$ for all $m\in \barm$, irrespective of the value of $n$. 
Therefore, if we consider any four elements $m_1,m_2,m_3,m_4\in \barm$ arranged in a $2\times 2$ matrix
$$
x=\left(\begin{array}{cc} m_1 & m_2\\ m_3 & m_4 \end{array}\right), 
$$
we have $x^4 = 0$, so the maximal order of elements in $\RR_u(G)$ is $4$ in this situation.
However, we can make the minimal $n$ such that $\barm^n = 0$ as large as we like (for example, 
by letting $k$ be a field of rational functions in several variables
$T_1,T_2,\ldots$ and then adjoining elements $t_i$ with $t_i^2 = T_i$ for each $i$). 

(ii). Suppose $G = \R_{k'/k}(\Gm)$.
If the exponent of the extension $k'/k$ is $e$, then because the $p^e$ power map sends $(k')^\times$ to $k^\times$,
we see that elements of $\RR_u(G)$ have order at most $p^e$, and some elements do indeed have this order (by the argument for ${\rm GL}_r$ above, taking $r=1$).
Hence, in this case, the exponent of $\RR_u(G)$ coincides with the exponent of the extension $k'/k$.
\end{example} 

Motivated by the examples above, which show dependence on the exponent of the extension $k'/k$, 
we finish with a bound for matrices which depends on that exponent and the rank of the matrices, rather than the number $n$.

\begin{lemma}\label{lem:rankbound}
Suppose $k'/k$ has exponent $e$, and let $G = \R_{k'/k}(\GL_r)$.
Then the exponent of $\RR_u(G)$ is at most $s$, 
where $s$ is chosen so that $p^s\geq r^2(p^e-1)$.
In particular, if $p\geq r^2$, then this exponent is at most $e+1$.
\end{lemma}

\begin{proof}
As before, we are after a power $p^s$ which kills all  $r\times r$ matrices $M$ with entries in $\barm$.
Since the exponent is $e$, it is true that $m^{p^e} = 0$ for all $m\in \barm$.
Again observe that a power $M^{p^s}$ has entries which are homogeneous polynomials of degree $p^s$
in the $r^2$ entries of $M$.
If we take $p^s$ large enough so that any monomial of degree $p^s$ in $r^2$ variables must contain at least one of the variables at least $p^e$ times, then we know that $M^{p^s} = 0$.
The claimed bound now follows.
\end{proof}

\begin{remarks}
A few comments on this last result:
\begin{itemize}
\item[(i)] When $p=2$, $e=1$ and $r=2$ in Lemma \ref{lem:rankbound}, we obtain the bound $4 = p^{e+1}$ which was observed in Example \ref{ex:rankbound}(i). 
On the other hand, for larger $e$ we have to go up to $p^{e+2}$.

\item[(ii)] When $r=1$ in Lemma \ref{lem:rankbound}, we obtain the bound $p^e$ which was already observed in Example \ref{ex:rankbound}(ii).

\item[(iii)] The bound from Lemma \ref{lem:nbound} will obviously be much better in some cases than that provided by Lemma \ref{lem:rankbound}. The tension between these two observations seems to be something which merits further study, and we will return to it in future work.

\item[(iv)] Computer calculations suggest that if $G:=\R_{k'/k}(G')$ then the exponent of $\RR_u(G_{\bar k})$ is always the same as the exponent of $\RR_u(G_{\bar k})\cap B_{\bar k}$, where $B:=\RR_u(\R_{k'/k}(B'))$ for $B'=T'\ltimes U'$ a Borel subgroup of $G'$ with split maximal torus $T'$ (recall we assumed $k$ to be separably closed). We do not know a proof that these exponents are the same---assuming, of course, that it is even true. In any case, one can make a little more conceptual progress in obtaining an upper bound for the exponent of $\RR_u(G_{\bar k})\cap B_{\bar k}$, as follows.

The group $B$ is connected and solvable (since $B'$ is), so any unipotent subgroup of $B_{\bar k}$ is contained in $\RR_u(B_{\bar k})$.   In particular, $\RR_u(G_{\bar k})\cap B_{\bar k}$ is contained in $\RR_u(B_{\bar k})$.  We now give an upper bound for the exponent of the latter.  Observe that $B\cong C\ltimes U$, where $C:=\R_{k'/k}(T')$ and $U:=\R_{k'/k}(U')$.  Then $\RR_u(B_{\bar k})=\RR_u(C_{\bar k})\ltimes U_{\bar k}$. Now from Example~\ref{ex:rankbound}(ii) we get that the exponent of $\RR_u(C_{\bar k})$ is precisely $e$, where $e$ is the exponent of $k'/k$. Thus the $p^e$-power map sends $\RR_u(B_{\bar k})$ into $U_{\bar k}$; denote the image of this map by $J$. One knows that the exponent of $U'$ is the smallest integer $s$ such that  $p^s>h-1$, where $h$ is the Coxeter number of $G'$; see \cite[Order Formula 0.4]{Tes95}. Thus the $p^s$-power map on $U'$ factors through $1$, and so the $p^s$-power map on $U=\R_{k'/k}(U')$ also factors through $1$. Since $J\subseteq U_{\bar k}$, the exponent of $J$ is at most $s$. So one gets an upper bound of $e+s$ for the exponent of $\RR_u(B_{\bar k})$ (and hence for the exponent of $\RR_u(G_{\bar k})\cap B_{\bar k}$), though this bound can be further improved with knowledge of the structure of $k'$.

Elementary calculations show that when $G'=\GL_2$, we have $J=1$ when $k'/k$ is primitive, giving $e$ as the exponent, and $J\neq 1$ when $k'/k$ is imprimitive; since all elements of $J$ have order $p$, we then get $e+1$ as the exponent. To explain what happens in the imprimitive case, find $t\in k'$ with exponent $e$ in $k'/k$; since $k'$ is imprimitive, we may find also $u\in k'\setminus k(t)$. One checks \[\begin{pmatrix}1+1\otimes t-t\otimes 1 & 1\otimes u-u\otimes 1\\0 & 1\end{pmatrix}^{p^e}=\begin{pmatrix}1 & (1\otimes u-u\otimes 1)(1\otimes t-t\otimes 1)^{p^e-1}\\0 & 1\end{pmatrix}\neq 1.\] 
More generally we ask
\begin{question}\label{opq} Suppose $G'$ is not unusual and let $r=\log_p([k':k(k')^p])-1$. Let $s=\max\{\lceil\log_p(h_{L'}-1)\rceil\}$ where the maximum is over all Coxeter numbers $h_{L'}$ of Levi subgroups $L'$ of $G'$ with rank at most $r$. Is the exponent of $\RR_u(B_{\bar k})$ equal to $\max\{e+s,\lceil\log_p(h_{G'}-1)\rceil\}$?\end{question}

\end{itemize}
\end{remarks}

\begin{remark}
 Note that the bounds above fail in the context of a general standard pseudo-reductive group. For instance, given a finite purely inseparable field extension $k'/k$ with exponent $e$, we can take $G$ to be the standard group arising from the 4-tuple $(G',k'/k,T',C)$, where $G'= T'= 1$ and $C= \R_{k''/k}(\Gm)$, with $k''/k$ a finite purely inseparable field extension of arbitrarily large exponent $f$: for then $G\cong C$ has exponent $f$ by Example~\ref{ex:rankbound}(ii).
\end{remark}

\bigskip
{\bf Acknowledgements}:
Part of the research for this paper was carried out while the
authors were staying at the Mathematical Research Institute
Oberwolfach supported by the ``Research in Pairs'' programme.
The authors also acknowledge the financial support of EPSRC Grant EP/L005328/1 and Marsden Grant UOA1021.

We thank Brian Conrad for comments and discussion. We thank Gopal Prasad for some helpful hints to improve the results.  We are also grateful to the referee for their comments.

\providecommand{\bysame}{\leavevmode\hbox to3em{\hrulefill}\thinspace}
\providecommand{\MR}{\relax\ifhmode\unskip\space\fi MR }
% \MRhref is called by the amsart/book/proc definition of \MR.
\providecommand{\MRhref}[2]{%
  \href{http://www.ams.org/mathscinet-getitem?mr=#1}{#2}
}
\providecommand{\href}[2]{#2}

\end{document}